\documentclass[oneside]{amsart}
\usepackage{amssymb}
\usepackage{amsthm}
\usepackage[dvipsnames]{xcolor}
\usepackage[colorlinks = true,
linkcolor = Blue,
urlcolor  = Blue,
citecolor = Blue,
anchorcolor = Blue]{hyperref}
\usepackage{cleveref}
\usepackage{tikz-cd} 
\usepackage{todonotes}
\usepackage{xstring}
\usepackage[mathscr]{euscript}   
\usepackage[normalem]{ulem}
\newcommand{\cat}[1]{
	\StrLen{#1}[\mystrlen]
	\ifnum\mystrlen=1 \mathscr{#1}
	\else \mathrm{#1}
	\fi}

\DeclareMathOperator{\NS}{NS}

\DeclareMathOperator{\Pic}{Pic}

\DeclareMathOperator{\Gal}{Gal}

\theoremstyle{definition}
\newtheorem{definition}{Definition}[section]

\theoremstyle{plain}
\newtheorem{theorem}[definition]{Theorem}
 
\newtheorem{lemma}[definition]{Lemma}

\newtheorem{mainthm}{Theorem}
\newtheorem{remark}{Remark}

\begin{document}
\title{The  Tate conjecture for ordinary  Calabi--Yau threefolds}
	
	\author[Lukas Brantner]{Lukas Brantner}
	\address{Oxford University,   CNRS (Universit\'{e} Paris-Saclay)} 
	\email{lukas.brantner@maths.ox.ac.uk}

	\begin{abstract}
We use the theory of canonical lifts to prove the integral geometric Tate conjecture for ordinary Calabi--Yau threefolds over finite fields, and also for  divisor classes on certain ordinary Calabi--Yau varieties.
	\end{abstract}
	\
\vspace{-8pt} 
	\maketitle

 	Let $X$ be a smooth projective variety over a finite field
 	$k$ of \mbox{characteristic $p$.} Given a  prime $\ell \neq p$, the \textit{Tate conjecture}, an arithmetic analogue of the Hodge conjecture, 
 	predicts that  any ${\Gal(\overline{k}/k)}$-fixed class in $\mathrm{H}^{2i}(X_{\overline{k}}, \mathbf{Q}_{\ell}(i)) $  is a linear combination of  codimension $i$ \mbox{subvarieties of $X$, } i.e.\ that  the cycle class map 
 	$$\mathrm{cl}_i: \mathrm{CH}^i(X) \otimes_{\mathbf{Z}} \mathbf{Q}_{\ell}  \longrightarrow \mathrm{H}^{2i}_{\acute et} (X_{\overline{k}}, \mathbf{Q}_{\ell} (i))^{{\Gal(\overline{k}/k)}}$$
 	is surjective.   
 It was proven for divisor classes on 
  abelian varieties by Tate \cite{tate1966endomorphisms},  and for K3 surfaces, in increasing generality, by Artin--Swinnerton-Dyer \cite{artin1973shafarevich}, Nygaard \cite{nygaard1983tate}, Nygaard--Ogus \cite{nygaard1985tate}, Maulik \cite{maulik2014supersingular}, Charles \cite{charles2013tate}, Madapusi Pera \cite{madapusi2015tate},     Kim--Madapusi Pera \cite{kim20152, err}, and 
 Ito--Ito--Koshikawa \cite{ito2021cm}.   Moonen \cite{Moonen}, Hamacher--Yang--Zhao \cite{hamacher2025tate}, and  Guo--Yang \cite{guo2025tate} proved many further cases of the Tate conjecture for divisors in the \mbox{$h^{2,0}=1$ setting. }
 
  For abelian threefolds, Totaro  \cite{totaro2021integral} established the  \textit{integral geometric Tate conjecture}   (proposed by Schoen \cite{Schoen}), \mbox{showing  that}
 $$\mathrm{cl}_i: \mathrm{CH}^i(X_{\overline{k}} ) \otimes_{\mathbf{Z}} \mathbf{Z}_{\ell}  \longrightarrow \operatorname*{colim}_{l} \  \mathrm{H}^{2i}_{\acute et} (X_{\overline{k}}, \mathbf{Z}_{\ell} (i))^{\Gal(\overline{k}/l)} $$
 is surjective for all $\ell \neq p$; here   $l $ ranges over all   fields $k \subset l \subset \overline{k}$ with $l/k$   finite.   Moreover, Totaro extended   Voisin's celebrated work on the integral Hodge conjecture  \cite{voisin2006integral}, establishing  the  integral geometric Tate conjecture  for Calabi--Yau threefolds  over finitely generated fields \textit{in characteristic zero}.  \vspace{1pt}

In this note, we  prove the  integral geometric Tate conjecture for ordinary Calabi--Yau threefolds in positive characteristic. To this end, we apply the 
canonical lifts we   constructed  {with Taelman \cite{CYDef}} to 
 transfer results from characteristic $0$ to characteristic $p$. We expect that the approach taken here  will   also prove useful for further cases of the Tate conjecture for Calabi--Yau varieties. \vspace{1pt}

 For us, a Calabi--Yau variety over $k$ is a smooth, projective, and geometrically
 irreducible $k$-scheme $X$ with $\omega_{X/k} \cong \mathcal{O}_X$.
We call $X$ \textit{ordinary} if its Hodge- and Newton polygons agree, and  \textit{$n$-ordinary} if the first $n$ slopes of the two polygons agree (cf.\ \cite[Remark 1.3]{CYDef}). We then prove:
 	\begin{mainthm} \label{threefoldtatemainl}
	Let $X$ be a Calabi--Yau threefold over a finite field $k$ which satisfies one of the following conditions:
	\begin{enumerate}
		\item $X$ is  ordinary;
		\item $X$ is $1$-ordinary and strict, i.e.\   $\mathrm{H}^1(X,\mathcal{O}_X) = \mathrm{H}^2(X,\mathcal{O}_X) = 0$. 
	\end{enumerate}
	Then $X$ satisfies  the integral geometric Tate conjecture.
\end{mainthm}
The second case uses a recent lifting result of Ward \cite{Ward}, whose proof relies on the crystalline period map constructed in \cite{CYDef}. 

 In general dimensions, we offer the following result:
\begin{mainthm}\label{divisortate}
	Let $X$ be a $2$-ordinary Calabi--Yau variety with   $\mathrm{H}^2(X, \mathcal{O}_X) = 0$.
	Then $X$ satisfies the integral geometric Tate conjecture for divisor classes.
\end{mainthm}

\begin{remark}
We can replace the ordinarity assumption in Theorem A(1) and the $2$-ordinarity assumption in Theorem B by the condition that $X$ is Bloch--Kato \mbox{$2$-ordinary} with $\mathrm{H}^{\dim(X)}_{\acute et}(X_{\overline{k}}, \mathbf{Z}_p)$ torsion-free. 
\end{remark}

 \subsection*{Methods}
 The  results of this   short note are obtained by combining 
several   recent advances  in the field. The central idea is that the canonical lifts of ordinary Calabi--Yau varieties  we  constructed with Taelman \cite{CYDef}, as well as Ward's lifting result for strict $1$-ordinary Calabi--Yau threefolds \cite{Ward}, allow one to generalise  earlier approaches to the Tate conjecture for K3 surfaces.

Taking inspiration from the works of Voisin \cite{voisin2006integral}  and Totaro \cite{totaro2021integral}   on the Hodge conjecture, we treat the cases  $b_1(X)= 0$ and  $b_1(X)>0$ separately.

If   $b_1(X)$ vanishes, the statement for divisors follows by 
applying the  Lefschetz $(1,1)$-theorem to our lift and   specialising. The claim for curve classes is proven by  applying Voisin's theorem \cite{voisin2006integral} to the lift and   specialising.
 
If $b_1(X)>0$, the argument is  more subtle. 
To handle the case of divisors, we  use 
Patakfalvi--Zdanowicz's characteristic $p$ generalisation of the Beauville--Bogomolov decomposition \cite{patakfalvi2026weak} to reduce  to the (known) Tate conjecture for K3 surfaces and abelian threefolds. The analogous reduction in characteristic $0$   \mbox{appears in \cite{totaro2021integral}.}
 To handle the case of curve classes, we first use hard Lefschetz and the divisor case to deduce that given a Tate class $u$ in degree $4$, some multiple $Nu$ is   algebraic. We then apply our result on 
 lifting line bundles  to canonical lifts  in   \cite{CYDef}, 
  a saturation argument of Totaro \cite{totaro2021integral}, and his generalisation of Voisin's theorem to prove that the class $u$  is itself algebraic.

	\subsection*{Acknowledgements}
	We thank Matthew Ward for bringing his  recent preprint \cite{Ward} to our attention, Ananth Shankar for commenting on an earlier draft, 	and   Piotr Achinger, Sam Moore, Noé Sotto, and Lenny Taelman for many exchanges about Calabi--Yau varieties \mbox{in characteristic $p$.}  We were  
supported by a Royal Society University Research Fellowship (URF\textbackslash R1\textbackslash 211075) and the \mbox{University of Oxford.}
 		\subsection*{AI Statement}
 We devised the  overall proof strategy. \mbox{ChatGPT} was used to explore the literature and review earlier drafts. The model directed our attention to the work of Patakfalvi and Zdanowicz, which helped us handle the case of divisor classes when $b_1(X)>0$. We independently checked all original sources used in our proofs and take full responsibility for the content of this article.
   
 \section{Proofs}

The  Hodge conjecture predicts that if $Y$ is a smooth complex projective  variety, then every class $
u \in \mathrm{H}^{2k}(Y,\mathbf{Q}) \cap \mathrm{H}^{k,k}(Y,\mathbf{C})$ is a rational linear combination of codimension $k$  cycles on $Y$.
While this conjecture is open in general, the   case $k=1$ is a theorem due to Lefschetz \cite{lefschetz1924analysis}, in fact   even   integrally:
\begin{theorem}[Lefschetz $(1,1)$-theorem] 
	Let $Y$ be a smooth complex projective variety (or a compact K\"{a}hler manifold).
	The cycle class map $$ \mathrm{Pic}(Y) \longrightarrow \mathrm{H}^{2}(Y,\mathbf{Z}) \cap \mathrm{H}^{1,1}(Y,\mathbf{C})$$
	is surjective.
\end{theorem}
When $\mathrm{H}^2(Y,\mathcal{O}_Y)=0$, the proof is particularly simple: by Hodge symmetry, any element of $\mathrm{H}^{2}(Y,\mathbf{Z})$
is a Hodge cycle, and  the exponential sequence 
 $$ 0 \rightarrow \mathbf{Z} \rightarrow \mathcal{O}_Y \rightarrow \mathcal{O}_Y^\times \rightarrow 0$$  
 induces a surjection
 $\mathrm{H}^{1}(Y,\mathcal{O}_Y^{\times}) \xrightarrow{c_1}  \mathrm{H}^{2}(Y,\mathbf{Z}), $
i.e.\ an isomorphism between $\mathrm{H}^{2}(Y,\mathbf{Z})$ and the 
   Néron–Severi group $\NS(Y) = \Pic(Y)/\Pic^0(Y).$

 We record a standard consequence of the Lefschetz $(1,1)$-theorem and \mbox{specialisation:}
\begin{lemma}\label{liftgivesdivisors}
Let $X$ be a smooth projective variety 
 over a  field $k \subseteq \overline{\mathbf{F}}_{p}$.
Assume that 
$X$ admits a smooth projective lift $\mathcal{X}$  to the Witt vectors $W(k)$
satisfying $\mathrm{H}^2(\mathcal{X}_K, \mathcal{O}_{\mathcal{X}_K})  = 0 $ for $K = \mathrm{Frac}(W(k))$.	Then $X$ satisfies the integral geometric Tate conjecture on codimension-1 cycles.
\end{lemma}
\begin{proof} 
 Maulik--Poonen's refinement \cite[Proposition 3.6]{zbMATH06073742} of Grothendieck's specialisation results  
	\cite[Expos\'{e} X, Appendice 7]{ferrand2006theorie} gives a commutative diagram
	$$ \begin{tikzcd}
		\NS(\mathcal{X}_{\overline{K}}) \otimes \mathbf{Z}_{\ell} \arrow[r,hook,"c_1"] \arrow[d,"\mathrm{sp}"] & \mathrm{H}^2_{\acute et} (\mathcal{X}_{\overline{K}}, \mathbf{Z}_\ell(1))\ \  \arrow[d,"\cong"] 
		\\ 
		\NS(X_{\overline{k}}) \otimes \mathbf{Z}_{\ell} \arrow[r,hook, "c_1"] & \mathrm{H}^2_{\acute et} (X_{\overline{k}}, \mathbf{Z}_\ell(1)) \ .
	\end{tikzcd}$$
	To verify the claim, it suffices to check that the top horizontal map is an isomorphism, and by \cite[Lemma 6.2]{totaro2021integral}, it is enough to do this
	 rationally (note that $\mathcal{X}_{\overline{K}}$ descends to some  subfield finitely generated over $\mathbf{Q}$).
This, in turn, is a consequence of the Lefschetz $(1,1)$-theorem (cf.\ e.g.\  \cite[Proposition 2.3]{gregory2019p}).
\end{proof}

We can now deduce our \Cref{divisortate} on divisor classes.
\begin{proof}[Proof of \Cref{divisortate}]
	By \cite[Theorem C]{CYDef}, $X$ admits a canonical lift $\mathcal{X}$ to $W(k)$, and the claim follows from Lemma \ref{liftgivesdivisors} 
	as upper semicontinuity (cf.\ \cite[0BDN]{stacks-project}) implies $\mathrm{H}^2(\mathcal{X}_K, \mathcal{O}_{\mathcal{X}_K}) = 0$.	 
\end{proof}

For strict Calabi--Yau threefolds, the existence of lifts has a striking consequence:
 	 	 	\begin{lemma} \label{threefoldtategenerallemma}
	Let $X$ be a strict Calabi--Yau threefold  over a field $k \subseteq \overline{\mathbf{F}}_{p}$
 and assume $X$
 admits a smooth projective lift $\mathcal{X}$ to $W(k)$.
	Then $X$ satisfies the integral geometric Tate conjecture.
\end{lemma}

\begin{proof} Let us write $K = \mathrm{Frac}(W(k))$ and fix an 
	embedding  $\overline{K}  \hookrightarrow \mathbf{C}$ throughout.
	We begin by noting that since $X$ is strict, line bundles lift uniquely (cf.\ e.g.\ \cite[Lemma 26.15]{gortz2023algebraic}), which implies that $\mathcal{X}$ has trivial canonical bundle.
	
	The case $i=1$ corresponding to divisor classes follows immediately from Lemma \ref{liftgivesdivisors}, as we have $\mathrm{H}^2(\mathcal{X}_K, \mathcal{O}_{\mathcal{X}_K}) = 0$ by upper  semicontinuity.
	
To verify the case  $i=2$, 
let  us pick a class $u \in \mathrm{H}^{4}_{\acute et} (X_{\overline{k}}, \mathbf{Z}_{\ell} (2))$.
By smooth proper base change,  $u$ lifts to  a class $\overline{u}$ in $\mathrm{H}^{4}_{\acute et} (\mathcal{X}_{\overline{K}}, \mathbf{Z}_{\ell} (2)) \cong 
\mathrm{H}^{4} (\mathcal{X}_{\mathbb{C}},\mathbb{Z}) \otimes  \mathbf{Z}_{\ell}$.
As $ \mathrm{H}^1(\mathcal{X}_{\mathbb{C}}, \Omega^3_{\mathcal{X}_{\mathbb{C}}}) =\mathrm{H}^1(\mathcal{X}_{\mathbb{C}}, \mathcal{O}_{\mathcal{X}_{\mathbb{C}}})  = 0$ by upper semicontinuity and flat base change, Hodge symmetry implies that $\overline{u}$
is a $\mathbf{Z}_{\ell}$-linear combination of Hodge classes in
$\mathrm{H}^4(\mathcal{X}_{\mathbb{C}}, \mathbb{Z})$.  
By Voisin's theorem \cite{voisin2006integral}, it is therefore
a $\mathbf{Z}_{\ell}$-linear combination of curve classes on $\mathcal{X}_{\mathbb{C}}$. By the  proof of \cite[Proposition 1.5]{deligne1982hodge}, we  may assume that  
these lie on  \mbox{$\mathcal{X}_{\overline{K}}$.}
\mbox{The result follows by specialising (cf.\  e.g.\  \cite[Example 20.3.5]{Fulton1998}).}

 The cases $i=0$ and $i=3$ are immediate.
\end{proof}

 We can now proceed to the proof of the main theorem. \vspace{-5pt} 
	
	\begin{proof}[Proof of \Cref{threefoldtatemainl}]
	Let us first assume that  $X$ is Bloch--Kato $2$-ordinary  with 
	$\mathrm{H}^3_{\acute et}(X_{\overline{k}}, \mathbf{Z}_p)$ torsion-free. By 
	\cite[Corollary 1.7, Theorem C]{CYDef}, $X$ admits a canonical lift  $\mathcal{X}$ over $W(k)$.
As the Newton polygon of its special fibre lies above the Hodge polygon of its generic fibre (cf.\ e.g.\ \cite[8.3.6]{arthur1994f}),  $\mathcal{X}_K$ is  also a \mbox{Calabi--Yau variety.}
	 
	If $b_1(X)=0$, then $b_1(\mathcal{X}_{\mathbb{C}})=0$, and 
	so Hodge theory and 
	 Serre duality imply that $\mathcal{X}_{\mathbb{C}}$ is strict. The 
	same argument as in Lemma \ref{threefoldtategenerallemma} then implies the result.
	
So let us assume that  $b_1(X)>0$ and  start with the case of divisor classes. By the proof of  \cite[Theorem 12.8]{patakfalvi2026weak}, there is  a \mbox{generically finite dominant morphism} $$ Z \longrightarrow X_{\overline{k} } $$
with $Z$ either a product of an elliptic curve with a $K3$ surface or an abelian threefold. Here we have used that 
$X_{\overline{k}}$ is $F$-split by  \cite[Proposition 3.1.2]{achinger2018serre}.

By \cite[Lemma 6.2]{totaro2021integral} and 
\cite[Theorem 5.2]{tate1994conjectures} (see also \cite[Lemma 2.2.1]{ambrosi2019note}), it suffices to verify the rational Tate conjecture for divisors on abelian threefolds as well as products of elliptic curves with $K3$ surfaces.
But these statements follow from Tate's result on divisor classes on abelian varieties and  
the more recent works on K3 surfaces cited in the introduction.

		The case  $i=2$ of curve classes is more subtle; our proof is a variant of the saturation argument in   \cite[Lemma  6.3]{totaro2021integral}.
		Let us pick an ample line \mbox{bundle $\mathcal{L}$ on $X$.}
	Given  a   class $u$ in $ \mathrm{H}^{4}_{\acute et} (X_{\overline{k}}, \mathbf{Z}_{\ell} (2))$ which is fixed by $\mathrm{Gal}(\overline{k}/l)$ for some finite field extension $l/k$, 
	write $\overline{u}$ 
	for the corresponding class 
	in $\mathrm{H}^{4}_{\acute et} (\mathcal{X}_{\overline{K}}, \mathbf{Z}_{\ell} (2))$. We  can use hard Lefschetz to write \vspace{-2pt} $$ N u =  v \cup H,\vspace{2pt} $$ 
		where $H= c_1(\mathcal{L}) \in\mathrm{H}_{\acute et}^2(X_{\overline{k}}, \mathbf{Z}_{\ell}(1))$, $v \in \mathrm{H}_{\acute et}^2(X_{\overline{k}}, \mathbf{Z}_{\ell}(1))$, and $N$ is a positive integer.

	As cupping with $H$ is Galois-equivariant, 
	we may,
 after possibly replacing $l$ by a finite extension  as in the proof of \cite[Lemma 6.3]{totaro2021integral},
  assume  that 
	$v$ is fixed by $\mathrm{Gal}(\overline{k}/l)$.
The claim for divisors then shows that $v$ is algebraic, and  so we can    write  
$$v= \sum_i \lambda_i c_1(\mathcal{L}_i) \vspace{-1pt}  $$		
		for suitably chosen line bundles $\mathcal{L}_1, \ldots, \mathcal{L}_n $ on $X_{\overline{k}}$ and scalars $\lambda_1, \ldots, \lambda_n \in \mathbf{Z}_{\ell}$.
We then use  \cite[Theorem  8.28]{CYDef}
to lift these line bundles    
to corresponding line bundles $\overline{\mathcal{L}}, \overline{\mathcal{L}}_1,\ldots, \overline{\mathcal{L}}_n$ on the generic fibre $\mathcal{X}_{\overline{K}}$.
The classes   $$\overline{H} = c_1(\overline{\mathcal{L}})  \ \ \ \ \ \ \ \ \ \ \ \ \ \overline{v} =  \sum_i \lambda_i c_1(\overline{\mathcal{L}}_i)  \vspace{-2pt}  \vspace{-3pt}   $$
  lift $H$ and $v$, and they are  both algebraic. Hence
$N\overline{u} = \overline{v} \cup \overline{H}$  is   a  $\mathbf{Z}_\ell$-linear combination of classes
of subvarieties of $\mathcal{X}_{\overline{K}}$. 
The image of $N\overline{u}$ under the isomorphism $$\mathrm{H}^{4}_{\acute et} (\mathcal{X}_{\overline{K}}, \mathbf{Z}_{\ell} (2)) \xrightarrow{\ \cong\ } \mathrm{H}^{4}(\mathcal{X}_{\mathbf{C}}, \mathbf{Z}) \otimes  \mathbf{Z}_{\ell}$$ is therefore a $\mathbf{Z}_{\ell}$-linear combination of integral Hodge classes, and since $$\mathrm{H}^{4}(\mathcal{X}_{\mathbf{C}}, \mathbf{Z})/ (\mathrm{H}^{4}(\mathcal{X}_{\mathbf{C}}, \mathbf{Z}) \cap \mathrm{H}^{2,2}(\mathcal{X}_{\mathbf{C}}, \mathbf{C}))$$ is torsion-free, the class $\overline{u}$ itself must be a $\mathbf{Z}_\ell$-linear combination of Hodge classes. 

By Totaro's refinement \cite[Theorem 4.1]{totaro2021integral} of Voisin's  Theorem 2 in  \cite{voisin2006integral}, this implies   
that $\overline{u}$ is a 
$\mathbf{Z}_{\ell}$-linear  combination of curve classes on $\mathcal{X}_{\mathbb{C}}$. As in 
the proof of Lemma \ref{threefoldtategenerallemma},  specialising implies that $u$ is algebraic. 

This completes the proof of case (1), as the claim is immediate \mbox{for $i=0$ and $i=3$.}

 If $X$ is strict and $1$-ordinary as in case (2), \cite[Theorem A]{Ward} shows that $X_{\overline{k}}$ admits a smooth projective lift to   $W(\overline{k})$. The claim follows from Lemma \ref{threefoldtategenerallemma}.   
\end{proof}

 \subsection*{Note added}
After the results of this article were formulated and proved, OpenAI released a manuscript \cite{OAI26} on the rational Hodge conjecture for CM abelian varieties and its consequences for the Tate conjecture for abelian varieties over finite fields. To the best of our knowledge, these
 results do not overlap with the new results we establish in this article.

	\bibliographystyle{amsalpha}
	\bibliography{There}

\end{document}